\documentclass[12pt,reqno]{amsart}
\usepackage{latexsym}
\usepackage{amssymb}
\usepackage{amscd}

\numberwithin{equation}{section}

\newtheorem{theorem}{Theorem}

\newtheorem{lemma}[theorem]{Lemma}

\theoremstyle{remark}

\def\ga{\gamma}

\def\fl#1{\left\lfloor#1\right\rfloor}

\begin{document}
\title[Congruences for the number of unique path
  partitions]{Congruences modulo powers of $2$ for the number of unique path
  partitions}
\author[C. Krattenthaler]{C. Krattenthaler} 

\address{$^{\dagger*}$Fakult\"at f\"ur Mathematik, Universit\"at Wien,
Oskar-Morgenstern-Platz~1, A-1090 Vienna, Austria.
WWW: {\tt http://www.mat.univie.ac.at/\lower0.5ex\hbox{\~{}}kratt}.}

\address{$^*$School of Mathematical Sciences, Queen Mary
\& Westfield College, University of London,
Mile End Road, London E1 4NS, United Kingdom.
}

\thanks{$^\dagger$Research partially supported by the Austrian
Science Foundation FWF, grants Z130-N13 and S50-N15,
the latter in the framework of the Special Research Program
``Algorithmic and Enumerative Combinatorics"}



\begin{abstract}
We compute the congruence class modulo 16 of the number of
unique path partitions of~$n$ (as defined by Olsson), thus
generalising previous results by Bessenrodt, Olsson and
Sellers [{\it Ann.\ Combin.} {\bf 13} (2013), 591--602].
\end{abstract}
\maketitle

\section{Introduction}

{\it Unique path partitions} were introduced by Olsson in \cite{OlssAA}.
Their study is motivated from the Murnaghan--Nakayama rule for the
calculation of the value of characters of the symmetric group.
They were completely characterised by Bessenrodt, Olsson and
Sellers in \cite{BeOSAA}. They used this characterisation to derive
a formula for the generating function for the number $u(n)$ of all
unique path partitions of~$n$. This formula reads
(cf.\ \cite[Remark~3.6]{BeOSAA})
\begin{align} \notag
\sum_{n\ge1}u(n)q^n&=2\sum_{i\ge1}q^{2^i-1}(1+q^{2^{i-1}})
\prod _{j=0} ^{i-2}\frac {1} {1-q^{2^j}}\\
&=
2\left(
{q(1+q)} 
+\sum_{i\ge2}\frac {q^{-1}+1} {1-q^2}\cdot
\frac {q^{2^i}(1+q^{2^{i-1}})} {\prod_{j=1}^{i-2}(1-q^{2^{j}})}
\right).
\label{eq:UGF}
\end{align}
The final part in \cite{BeOSAA} concerns congruences modulo~8 for $u(n)$.
The corresponding 
main result \cite[Theorem~4.6]{BeOSAA} provides a complete description
of the behaviour of $u(n)$ modulo~8 (in terms of the related sequence
of numbers $w(n)$; see the next section for the definition of~$w(n)$). 
The arguments to arrive at this result are mainly of a recursive nature.

The purpose of this note is to show that a more convenient and more
powerful method to derive congruences (modulo powers of~$2$ is by an analysis
of the generating function \eqref{eq:UGF}. Not only are we able to
recover the result from \cite{BeOSAA}, but in addition
we succeed in determining the congruence class of $u(n)$ modulo~16,
see \eqref{eq:wu} and Theorem~\ref{thm:w}, thus solving the problem
left open in the last paragraph of \cite{BeOSAA}.
We point out that the approach presented here is very much
inspired by calculations in \cite[Appendix]{KrMuAE}, where expressions
similar to the one on the right-hand side of \eqref{eq:UGF} appear,
with the role of the prime number~2 replaced by~3, though.

\section{An equivalent expression for the generating function}

We start with the observation (already made in \cite{BeOSAA}) that,
first, all numbers $u(n)$ are divisible by~$2$, and, second, we have
$u(2n)=u(2n-1)$ for all $n$. This is easy to see from the right-hand
side of \eqref{eq:UGF} since it has the form $2(1+q)f(q^2)$, where
$f(t)$ is a formal power series in~$t$.
We therefore divide the right-hand side of \eqref{eq:UGF} by $2(1+q^{-1})$,
subsequently replace $q$ by $q^{1/2}$,
and consider the ``reduced" generating function
\begin{align*} 
\notag
\sum_{n\ge2}w(n)q^n&=\sum_{i\ge2}q^{2^{i-1}}(1+q^{2^{i-2}})
\frac {1} {(1-q)\prod _{j=0} ^{i-3}(1-q^{2^j})}.
\end{align*}
In other words, we have 
\begin{equation} \label{eq:wu} 
2w(n)=u(2n)=u(2n-1)
\end{equation}
for all~$n$.

Using the convention
\begin{equation} \label{eq:SUM} 
\sum _{k=M} ^{N-1}\text {\rm Expr}(k)=\begin{cases} 
\hphantom{-}
\sum _{k=M} ^{N-1} \text {\rm Expr}(k),&N>M,\\
\hphantom{-}0,&N=M,\\
-\sum _{k=N} ^{M-1}\text {\rm Expr}(k),&N<M.\end{cases}
\end{equation}
for sums, we rewrite the above equation in the following way:
\begin{align} 
\notag
\sum_{n\ge2}w(n)q^n&=\sum_{i\ge2}q^{2^{i-1}}
\frac {1+q^{2^{i-2}}} {(1-2q+q^2)\prod _{j=1} ^{i-3}(1-q^{2^j})}\\
\notag
&=\sum_{i\ge2}q^{2^{i-1}}
\frac {1+q^{2^{i-2}}} 
{(1-\frac {2q} {1+q^2})(1+q^2)\prod _{j=1} ^{i-3}(1-q^{2^j})}\\
\notag
&=\sum_{i\ge2}q^{2^{i-1}}
\frac {1+q^{2^{i-2}}} 
{(1-\frac {2q} {1+q^2})(1-q^4)\prod _{j=2} ^{i-3}(1-q^{2^j})}\\
\notag
&=\sum_{i\ge2}q^{2^{i-1}}
\frac {1+q^{2^{i-2}}} 
{(1-\frac {2q} {1+q^2})(1-q^4)^2\prod _{j=3} ^{i-3}(1-q^{2^j})}\\
\notag
&\hphantom{{}={}}\hbox to 6cm{\leaders\hbox to .3cm{\hss.\hss}\hfill}\\
\notag
&=\sum_{i\ge1}q^{2^{2i-1}}
\frac {1+q^{2^{2i-2}}} 
{(1-\frac {2q} {1+q^2})(1-\frac {2q^4} {1+q^8})\cdots(1-\frac {2q^{2^{2i-4}}}
  {1+q^{2^{2i-3}}})
(1-q^{2^{2i-2}})}\\
\notag
&\kern1cm
+\sum_{i\ge1}q^{2^{2i}}
\frac {1} 
{(1-\frac {2q} {1+q^2})(1-\frac {2q^4} {1+q^8})\cdots(1-\frac {2q^{2^{2i-2}}}
  {1+q^{2^{2i-1}}})
}\\
&=\sum_{i\ge1}q^{2^{2i-1}}
\frac {1+\frac {2q^{2^{2i-2}}} {1-q^{2^{2i-2}}}} 
{\prod _{j=0} ^{i-2}(1-\frac {2q^{2^{2j}}}
  {1+q^{2^{2j+1}}})
}
+\sum_{i\ge1}q^{2^{2i}}
\frac {1} 
{\prod _{j=0} ^{i-1}(1-\frac {2q^{2^{2j}}}
  {1+q^{2^{2j+1}}})}
.
\label{eq:WGF2}
\end{align}

From the last expression it is immediately obvious that $w(n)$ is odd
if and only if $n$ is a power of~$2$, thus recovering the first
assertion of \cite[Cor.~4.3]{BeOSAA}. The above mentioned mod-4 result 
\cite[Theorem~4.6]{BeOSAA} for $w(n)$ --- which, by \eqref{eq:wu}, 
translates into a mod-8 result for~$u(n)$ --- 
can also be derived within a few lines from the above expression.

In the next section, we show how to obtain congruences modulo~$8$
for~$w(n)$, which, by \eqref{eq:wu}, translate into congruences
modulo~16 for the unique path partition numbers~$u(n)$.

\section{Congruences modulo powers of $2$}

In what follows, we write 
$$f(q)=g(q)~\text {modulo}~2^\ga$$ 
to mean that the coefficients
of $q^i$ in $f(q)$ and $g(q)$ agree modulo~$2^\ga$ for all $i$.
We apply geometric series expansion in \eqref{eq:WGF2}, and at the
same time we neglect terms which are divisible by~$8$. For example,
we expand
$$
\frac 1 {1-\frac {2q} {1+q^2}}=1+\frac {2q} {1+q^2}+
\frac {4q^2} {(1+q^2)^2}\quad 
\text{modulo }8.
$$
In this manner, we obtain the congruence
\begin{align} 
\notag
\sum_{n\ge2}w(n)q^n&=
\sum_{i\ge1}
q^{2^{2i-1}}\Bigg(
1+\frac {2q^{2^{2i-2}}} {1-q^{2^{2i-2}}}
+2\sum_{j=0}^{i-2}\frac {q^{2^{2j}}}
  {1+q^{2^{2j+1}}}\\
\notag
&\kern1.5cm
+4\frac {q^{2^{2i-2}}} {1-q^{2^{2i-2}}}
\sum_{j=0}^{i-2}\frac {q^{2^{2j}}}
  {1+q^{2^{2j+1}}}
+4\sum_{0\le s\le t\le i-2}^{}\frac {q^{2^{2s}+2^{2t}}}
  {(1+q^{2^{2s+1}})(1+q^{2^{2t+1}})}
\Bigg)\\
\notag
&\kern.5cm
+\sum_{i\ge1}q^{2^{2i}}
\Bigg(1
+2\sum_{j=0}^{i-1}\frac {q^{2^{2j}}}
  {1+q^{2^{2j+1}}}
+4\sum_{0\le s\le t\le i-1}\frac {q^{2^{2s}+2^{2t}}}
  {(1+q^{2^{2s+1}})(1+q^{2^{2t+1}})}
\Bigg)\\
\notag
&\kern11.5cm
\text{modulo }8.
\end{align}
After rearrangement, this becomes
\begin{align} 
\notag
\sum_{n\ge2}w(n)q^n&=
\sum_{i\ge1}q^{2^i}
+\frac {2q^3} {1-q}
+2\sum_{j\ge1}\frac {1}
  {1-q^{2^{2j}}}\Bigg(
q^{2^{2j}+2^{2j+1}}
+q^{2^{2j-2}}(1-q^{2^{2j-1}})\sum_{\ell\ge 2j}q^{2^\ell}
\Bigg)\\
\notag
&\kern1cm
+4\sum_{1\le s<t}\frac {q^{2^{2s-2}+2^{2t-2}}}
  {(1-q^{2^{2s-1}})(1-q^{2^{2t-1}})}
\Bigg(
q^{2^{2t-1}}(1+q^{2^{2t-2}})
+\sum_{\ell\ge2t}q^{2^\ell}
\Bigg)\\
&\kern1cm
+4\sum_{s\ge1}\frac {q^{2^{2s-1}}}
  {(1-q^{2^{2s}})}
\sum_{\ell\ge2s}q^{2^\ell}
\quad \quad \quad \quad 
\text{modulo }8.
\label{eq:Wcong1}
\end{align}

We must now analyse the individual sums in \eqref{eq:Wcong1}.

\begin{lemma} \label{lem:S1}
Let $n\ge2$, and write 
$n=\sum_{i=a}^e n_i\cdot 2^i$, with $0\le n_i\le1$ for all~$i$
and $n_a\ne0\ne n_e$.
Then the coefficient of $q^n$ in 
\begin{equation} \label{eq:S1} 
\sum_{j\ge1}\frac {q^{2^{2j}+2^{2j+1}}}
  {1-q^{2^{2j}}}
\end{equation}
is equal to $\fl{a/2}$ if $n$ is not a power of\/ $2$, and it is
equal to $\max\{\fl{a/2}-1,0\}$ otherwise.
\end{lemma}

\begin{proof}
By geometric series expansion, we see that the coefficient of $q^n$ in 
\eqref{eq:S1} is equal to the number of possibilities to write
$n=(k+3)2^{2j}$ for some $j\ge1$ and $k\ge0$. For fixed $j$, we can
find a suitable $k$ if and only if $n\ge 3\cdot 2^{2j}$.
If $n$ is not a power of $2$,
this is equivalent to the condition that $2j\le a$.
The claim follows immediately.
\end{proof}

\begin{lemma} \label{lem:S2}
Let $n\ge2$, and write 
$n=\sum_{i=a}^e n_i\cdot 2^i$ as in Lemma~\ref{lem:S1}.
Then the coefficient of $q^n$ in 
\begin{equation} \label{eq:S2} 
\sum_{j\ge1}
\frac {q^{2^{2j-2}}}
  {1-q^{2^{2j}}}
\sum_{\ell\ge 2j}q^{2^\ell}
\end{equation}
is equal to $e-2j+1$ if $a=2j-2$, $n_{a+1}=n_{2j-1}=0$, 
and $n$ is not a power of $2$, 
and it is equal to $0$ otherwise.
\end{lemma}

\begin{proof}
By geometric series expansion, we see that the coefficient of $q^n$ in 
\eqref{eq:S2} is equal to the number of possibilities to write
$n=2^{2j-2}+k\cdot 2^{2j}+2^\ell$ for some $j\ge1$, $\ell\ge2j$, and $k\ge0$. 
The claim follows immediately.
\end{proof}

\begin{lemma} \label{lem:S3}
Let $n\ge2$, and write 
$n=\sum_{i=a}^e n_i\cdot 2^i$ as in Lemma~\ref{lem:S1}.
Then the coefficient of $q^n$ in 
\begin{equation} \label{eq:S3} 
\sum_{j\ge1}
\frac {q^{2^{2j-2}+2^{2j-1}}}
  {1-q^{2^{2j}}}
\sum_{\ell\ge 2j}q^{2^\ell}
\end{equation}
is equal to $e-2j+1$ if $a=2j-2$, $n_{a+1}=n_{2j-1}=1$, 
and it is equal to $0$ otherwise.
\end{lemma}

\begin{proof}
By geometric series expansion, we see that the coefficient of $q^n$ in 
\eqref{eq:S3} is equal to the number of possibilities to write
$n=2^{2j-2}+2^{2j-1}+k\cdot 2^{2j}+2^\ell$ for some $j\ge1$, 
$\ell\ge2j$, and $k\ge0$. 
The claim follows immediately.
\end{proof}

\begin{lemma} \label{lem:S5}
Let $n\ge2$, and write 
$n=\sum_{i=a}^e n_i\cdot 2^i$ as in Lemma~\ref{lem:S1}.
Then the coefficient of $q^n$ in 
\begin{equation} \label{eq:S5} 
\sum_{s\ge1}
\frac {q^{2^{2s-1}}}
  {1-q^{2^{2s}}}
\sum_{\ell\ge 2s}q^{2^\ell}
\end{equation}
is equal to $e-2s+1$ if $a=2s-1$
and $n$ is not a power of $2$, 
and it is equal to $0$ otherwise.
\end{lemma}

\begin{proof}
By geometric series expansion, we see that the coefficient of $q^n$ in 
\eqref{eq:S5} is equal to the number of possibilities to write
$n=2^{2s-1}+k\cdot 2^{2s}+2^\ell$ for some $s\ge1$, $\ell\ge2s$, and $k\ge0$. 
The claim follows immediately.
\end{proof}

\begin{lemma} \label{lem:S4}
Let $n\ge2$, and write 
$n=\sum_{i=a}^e n_i\cdot 2^i$ as in Lemma~\ref{lem:S1}.
Then the coefficient of $q^n$ in 
\begin{equation} \label{eq:S4} 
\sum_{1\le s<t}\frac {q^{2^{2s-2}+2^{2t-2}}}
  {(1-q^{2^{2s-1}})(1-q^{2^{2t-1}})}
\sum_{\ell\ge2t-1}q^{2^\ell}
\end{equation}
is congruent to 
\begin{equation} \label{eq:S4erg}
e\sum_{i=a+2}^{e-\chi(e\text{ \em even})}
n_{i}
-a\cdot n_{a+2}
+\fl{\tfrac {1} {2}(e-a-1)}
\quad \quad 
\text{\em (mod~2)}, 
\end{equation}
where $\chi(\mathcal S)=1$ if $\mathcal S$ is
true and $\chi(\mathcal S)=0$ otherwise.
\end{lemma}

\begin{proof}
By geometric series expansion, we see that the coefficient of $q^n$ in 
\eqref{eq:S4} is equal to the number of possibilities to write
\begin{equation} \label{eq:k1k2a} 
n=(2k_1+1)2^{2s-2}+(2k_2+1)2^{2t-2}+2^{2t-1+k_3}
\end{equation}
for some $s$ and $t$ with
$1\le s<t$ and $k_1,k_2,k_3\ge0$.
Clearly, we need $a$ to be even in order that
the number of these possibilities be non-zero.
Given that $a=2s-2$,
we just have to count the number of possible triples $(t,k_2,k_3)$ in  
\eqref{eq:k1k2a}, since the appropriate $k_1$ can certainly be found.
If we fix $t$ and $k_3$, the number of possible $k_2$'s is
$$
\fl{\frac {1} {2}\cdot \frac {n-2^{2t-1+k_3}} {2^{2t-2}}+\frac {1} {2}}
=
\fl{\frac {n} {2^{2t-1}}+\frac {1} {2}}-2^{k_3}.
$$
This needs to be summed over all $t$ and $k_3$
with $\frac {1} {2}(a+2)=s<t\le \frac {1} {2}(e+1)$ and $0\le k_3\le e-2t+1$. 
We obtain
\begin{align*} 
\sum_{t=s+1}^{\fl{\frac {1} {2}(e+1)}}
&\sum_{k_3=0}^{e-2t+1}
\left(\fl{\frac {n} {2^{2t-1}}+\frac {1} {2}}-2^{k_3}\right)\\
&
\equiv
\sum_{t=\frac {1} {2}(a+4)}^{\fl{\frac {1} {2}(e+1)}}
\sum_{k_3=0}^{e-2t+1}
\big\lfloor
n_a\cdot 2^{a-2t+1}+\dots+(n_{2t-2}+1)\cdot 2^{-1}\\
&\kern4cm
+n_{2t-1}+n_{2t}\cdot 2+\dots +n_e\cdot 2^{e-2t+1}
\big\rfloor
-\fl{\tfrac {1} {2}(e-a-1)}\\
&
\equiv
\sum_{t=\frac {1} {2}(a+4)}^{\fl{\frac {1} {2}(e+1)}}
(e-2t+2)
(n_{2t-2}+n_{2t-1})
+\fl{\tfrac {1} {2}(e-a-1)}\\
&
\equiv
e\sum_{t=\frac {1} {2}(a+4)}^{\fl{\frac {1} {2}(e+1)}}
n_{2t-1}
+e\sum_{t=\frac {1} {2}(a+4)}^{\fl{\frac {1} {2}(e-1)}}
n_{2t}
+(e-a)n_{a+2}
+\fl{\tfrac {1} {2}(e-a-1)}
\quad \quad (\text{mod }2).
\end{align*}
\end{proof}

\begin{lemma} \label{lem:S5a}
Let $n\ge2$, and write 
$n=\sum_{i=a}^e n_i\cdot 2^i$ as in Lemma~\ref{lem:S1}.
Then the coefficient of $q^n$ in 
\begin{equation} \label{eq:S5a} 
\sum_{1\le s<t}\frac {q^{2^{2s-2}+2^{2t}}}
  {(1-q^{2^{2s-1}})(1-q^{2^{2t-1}})}
\end{equation}
is congruent to 
\begin{equation} \label{eq:S5aerg}
\sum_{t=\frac {1} {2}(a+4)}^{\fl{\frac {1} {2}(e+1)}}
n_{2t-1}
+\fl{\tfrac {1} {2}(e-a-1)}.
\quad \quad 
\text{\em (mod~2)}.
\end{equation}
\end{lemma}

\begin{proof}
By geometric series expansion, we see that the coefficient of $q^n$ in 
\eqref{eq:S5a} is equal to the number of possibilities to write
\begin{equation} \label{eq:k1k2a2} 
n=(2k_1+1)2^{2s-2}+(k_2+2)2^{2t-1}
\end{equation}
for some $s$ and $t$ with
$1\le s<t$ and $k_1,k_2\ge0$.
Clearly again, we need $a$ to be even in order that
the number of these possibilities be non-zero.
Given that $a=2s-2$,
we just have to count the number of possible pairs $(t,k_2)$ in  
\eqref{eq:k1k2a2}, since the appropriate $k_1$ can certainly be found.
If we fix $t$, the number of possible $k_2$'s is
$$
\fl{\frac {n-2^{2t}} {2^{2t-1}}+1}
=
\fl{\frac {n} {2^{2t-1}}}-1.
$$
This needs to be summed over all $t$
with $\frac {1} {2}(a+2)=s<t\le \frac {1} {2}(e+1)$.
We obtain
\begin{align*} 
\sum_{t=s+1}^{\fl{\frac {1} {2}(e+1)}}
\left(\fl{\frac {n} {2^{2t-1}}}-1\right)
&\equiv
\sum_{t=\frac {1} {2}(a+4)}^{\fl{\frac {1} {2}(e+1)}}
\big\lfloor
n_a\cdot 2^{a-2t+1}+\dots+n_{2t-2}\cdot 2^{-1}\\
&\kern3cm
+(n_{2t-1}-1)+n_{2t}\cdot 2+\dots +n_e\cdot 2^{e-2t+1}
\big\rfloor\\
&
\equiv
\sum_{t=\frac {1} {2}(a+4)}^{\fl{\frac {1} {2}(e+1)}}
n_{2t-1}
-\fl{\tfrac {1} {2}(e-a-1)}
\quad \quad (\text{mod }2).
\qedhere
\end{align*}
\end{proof}

We are finally in the position to state and prove our main result.
It expresses the congruence class of $w(n)$ modulo~8 --- and thus,
by \eqref{eq:wu}, the congruence class of the unique path partition
number~$u(n)$ modulo~16 --- in terms of the binary digits of~$n$.
We point out that the assertion \eqref{eq:w1} already appeared
in \cite[Prop.~4.5]{BeOSAA}.

\begin{theorem} \label{thm:w}
Let $n\ge2$, and write 
$n=\sum_{i=a}^e n_i\cdot 2^i$ as in Lemma~\ref{lem:S1}.
Then, if $a=e$ {\em(}i.e., if $n$ is a power of~$2${\em)}, 
the number $w(n)$ is congruent to 
\begin{equation} \label{eq:w1} 
2\fl{a/2}+1\quad \text{\em(mod~8)},
\end{equation}
while it is congruent to
\begin{multline}
2+2\fl{a/2}+2\chi(a\text{ \em even})(1-2n_{a+1})(e-a-1)
+4\chi(a\text{ \em odd})(e-a)\\
+4\chi(a\text{ \em even})\Bigg(e\sum_{i=a+2}^{e-\chi(e\text{ \em even})}
n_{i}
+a\cdot n_{a+2}
+\sum_{t=\frac {1} {2}(a+4)}^{\fl{\frac {1} {2}(e+1)}}
n_{2t-1}
\Bigg)\quad \text{\em(mod 8)}
\label{eq:w2}
\end{multline}
otherwise.
\end{theorem}

\begin{proof}
Let first $n=2^a$. We must then read the coefficient of $q^n$ on the
right-hand side of \eqref{eq:Wcong1} and reduce the result modulo~8.
Non-zero contributions come from the very first sum, from the series
$2q^3/(1-q^2)$, and from the series which is discussed in Lemma~\ref{lem:S1}.
Altogether, we obtain
$$1+2\chi(a\ge2)+2\max\{\fl{a/2}-1,0\},$$
which can be simplified to \eqref{eq:w1}. 

Now let $n$ be different from a power of $2$. 
The non-zero contributions when reading the coefficient of $q^n$ on
the right-hand side of \eqref{eq:Wcong1} come again from the series
$2q^3/(1-q^2)$, and from the series discussed in
Lemmas~\ref{lem:S1}--\ref{lem:S5a}. These contributions add up to
\begin{multline*}
2\chi(n\ge3)+2\fl{a/2}+2\chi(a\text{ even, $n_{a+1}=0$})(e-a-1)\\
+2\chi(a\text{ even, $n_{a+1}=1$})(e-a-1)
+4\chi(a\text{ odd})(e-a)\\
+4\chi(a\text{ even})\Bigg(e\sum_{i=a+2}^{e-\chi(e\text{ even})}
n_{i}
-a\cdot n_{a+2}+\fl{\tfrac {1} {2}(e-a-1)}
+\sum_{t=\frac {1} {2}(a+4)}^{\fl{\frac {1} {2}(e+1)}}
n_{2t-1}+\fl{\tfrac {1} {2}(e-a-1)}
\Bigg).\kern-4pt
\end{multline*}
This expression can be simplified to result in \eqref{eq:w2}.
\end{proof}

It is clear that, in the same way, 
one could also derive a result for $w(n)$ modulo~16, 32, \dots,
albeit at the cost of considerably more work.

\end{document}